\documentclass[11pt,draft]{article}

\usepackage[T1]{fontenc}
\usepackage{amsfonts,amssymb,amsmath,amsthm,amscd,enumerate}

\newtheorem{thm}{theorem}[section]
\newtheorem{theorem}[thm]{Theorem}
\newtheorem{proposition}[thm]{Proposition}
\newtheorem{lemma}[thm]{Lemma}

\newtheorem{remark}[thm]{Remark}

\newtheorem{definition}[thm]{Definition}

\begin{document}

\title{Graded Polynomial Identities for Matrices with the Transpose Involution over an Infinite Field}
\author{
Luís Felipe Gonçalves Fonseca
\thanks{\texttt{luisfelipe@ufv.br}}
\\
\\
Instituto de Ciências Exatas e Tecnlógicas \\
Universidade Federal de Viçosa \\
Florestal, MG, Brazil
\\
\\
Thiago Castilho de Mello \thanks{\texttt{tcmello@unifesp.br}}
\thanks{supported by Fapesp grant No. 2014/10352-4, Fapesp grant No.2014/09310-5, and CNPq grant No. 461820/2014-5. }
\\
\\
Instituto de Ci\^encia e Tecnologia\\
Universidade Federal de S\~ao Paulo\\
S\~ao Jos\'e dos Campos, SP, Brazil}

\maketitle

\begin{abstract}

    Let $F$ be an infinite field, and let $M_{n}(F)$ be the algebra of $n\times n$ matrices
    over $F$. Suppose that this algebra is equipped with an elementary grading whose neutral
    component coincides with the main diagonal. In this paper, we find a basis for the graded polynomial identities of $M_{n}(F)$ with the transpose involution. Our results generalize for infinite fields of arbitrary characteristic previous results in the literature which were obtained for the field of complex numbers and for a particular class of elementary G-gradings.
    
\end{abstract}

\section{Introduction}

Let $F$ be a field and $A$ be an $F$-algebra. A polynomial identity
of the algebra $A$ is a polynomial in noncommuting variables which
vanishes under any substitution of these variables by elements of
$A$.

One of the first important results about polynomial identities in
algebras is the Amitsur-Levitzki Theorem which proves that the
standard polynomial of degree $2n$ is a polynomial identity for
$M_n(F)$. Specht \cite{Specht} raised the question whether the
T-ideal of all polynomial identities of a given algebra is finitely
generated as a T-ideal. This problem was answered by Kemer
\cite{Kemer} using the characteristic zero some decades latter.
Although Kemer proved that there always exists a finite basis for
the identities of a given algebra in characteristic zero, it is very
difficult problem to find any such basis. The explicit polynomial
identities for concrete algebras are known in very few cases. For
instance, for $M_n(\mathbb{C})$ such a basis is known only if $n=1$
or $2$. In light of this, mathematicians started to work with
`weaker' polynomial identities such as identities with
trace, identities with involution and graded identities. It is worth mentioning that the graded identities were also used by Kemer in the solution of the Specht problem.

Subsequent to the pioneering work of Di-Vincenzo \cite{DiVincenzo}
about graded identities, many authors described the graded
identities of $M_{n}(F)$ and other different important algebras in
different contexts \cite{Azevedo}, \cite{Azevedo2},
\cite{BahturinDrensky}, \cite{Diogo}, \cite{Vasilovsky},
\cite{DiogoThiago}, \cite{CentroneMello} and \cite{Vasilovsky2}.
Also, identities with involution for $M_n(F)$ have been studied by
some authors \cite{Jones}.

Recently Haile and Natapov \cite{Haile} exhibited a basis for the
graded identities with involution of $M_n(\mathbb{C})$, endowed with
the transpose involution and with a crossed-product grading. This
grading is an elementary grading of $M_n(F)$ by a group
$G=\{g_1,\dots, g_n\}$ induced by the $n$-tuple $(g_1,\dots, g_n)$.
The authors used the graph theory as in \cite{Haile0}.

In this paper we generalize the results of Haile and Natapov for a broader class of
gradings and for infinite fields of arbitrary characteristic. The main tool here is the use of generic matrices,
and we use ideas similar to those in \cite{Azevedo},
\cite{Azevedo2}, \cite{CentroneMello}, \cite{Diogo}, and
\cite{DiogoThiago}.

\section{Preliminaries}

We denote by $F$ an infinite field of arbitrary characteristic. All vector spaces and algebras
are over $F$. We denote the algebra of $n\times n$ matrices over $F$
by $M_{n}(F)$ and a group with the unity $e$ by $G$.

If $A$ is an algebra and $G$ is a group, a $G$-grading on $A$ is a
decomposition of $A$ as a direct sum of subspaces $A=\oplus_{g\in
G}A_g$, indexed by elements of the group $G$, which satisfy
$A_gA_h\subseteq A_{gh}$, for any $g, h\in G$. If $a\in A_g -
\{0\}$, for some $g\in G$, we say that $a$ is homogeneous of degree
$g$ and we denote $\deg (a)=g$. The support of the grading, is the
subset of $G$, $\text{Supp}(A)=\{g\in G\,;\, A_g\neq \{0\}\}$.

If $1\leq i, j \leq n$, we denote by $e_{ij}$ the matrix with 1 on the position
$(i, j)$, and and 0 elsewhere. We call them \emph{elementary matrices}, or \emph{matrix units}.

Now let $(g_1,\dots, g_n)\in G^n$ be an $n$-tuple of elements of
$G$. For each $g\in G$, let $R_g\subseteq M_n(F)$ be the subspace
generated by the elementary matrices $e_{ij}$ for $i$ and $j$
satisfying $g_{i}^{-1}g_{j}=g$. Then $M_n(F)=\oplus_{g\in G} R_g$ is
a $G$-grading on $M_n(F)$ called \emph{elementary grading defined by
$(g_1,\dots,g_n)$}.

We recall a known result from \cite{Dascalescu}, which characterizes elementary gradings on $M_n(F)$.

\begin{theorem}
    If $G$ is any group, a $G$-grading of $M_{n}(F)$ is elementary if and only if all matrix units $e_{ij}$ are homogeneous.
\end{theorem}

An \emph{involution} on an algebra $A$ is an antiautomorphism of the
order two, that is, a linear map $^*:A\longrightarrow A$ satisfying
$(ab)^*=b^*a^*$ and $(a^*)^*=a$, for all $a,b\in A$. A classic
example of involution on $M_{n}(F)$ is the transpose map.

A $G$-graded algebra $A = \bigoplus_{g \in G} A_{g}$ with involution
$^*$ is called a \emph{degree-inverting involution algebra} if
$(A_{g})^{*} = A_{g^{-1}}$ for all $g \in G$. In this case, we say that $*$ is a degree-inverting involution on $A$. In this paper, if $A$ is a degree-inverting involution algebra, we say it is a
$(G,*)$-algebra. A typical example of a $(G,*)$-algebra is
$M_{n}(F)$ endowed with an elementary grading and with the transpose
involution. The degree-inverting involutions on $M_n(F)$ have been
described by the authors in \cite{FonsecaMello}.

\begin{remark}
When dealing with identities with involution on algebras over fields of characteristic different from 2, one usually consider the decomposition $A = A^{+}\oplus A^{-}$, where $A^{+} = \{a\in A\,|\, a^*=a \}$ (symmetric component) and $A^{-} = \{a\in A\,|\, a^*=-a \}$ (skew-symmetric component) and set in the free algebra, the set of symmetric and skew-symmetric variables. Note that one cannot use this approach in the present case, since the symmetric and skew-symmetric components are no longer homogeneous. In order to deal with our case, we need to consider a free algebra where the grading and the involution behave in the same way as in the algebra we want to study its identities. 
\end{remark}

\subsection{The free $(G,*)$-algebra and $(G,*)$-identities}

To describe the identities of $M_n(F)$ as a $(G,*)$-algebra, we
define what we call the \emph{free $(G,*)$-algebra}.

For each $g\in G$, we define two countable sets $X_{g} =
\{x_{k,g}\,;\, k\in \mathbb{N}\}$ and $X_{g}^{*} =
\{x_{k,g}^{*}\,;\,k\in \mathbb{N}\}$. Then, let $X=\cup_{g\in G}X_g$
and $X^*=\cup_{g\in G}X_g^*$. Consider the free associative algebra $F\langle X\cup X^*\rangle$, which is freely generated by $X\cup X^*$. Of course, it is an algebra with an involution defined in a natural way. Now, we define a $G$-grading on this free algebra to make it a $(G,*)$-algebra. Let $\deg(1)=e$, and for each $k\in \mathbb{N}$ and $g\in G$, let $\deg(x_{k,g})=g$ and $\deg(x_{k,g}^*)=g^{-1}$. If $m= x_{i_{1},g_{1}}^{\varepsilon_{1}}\cdots
x_{i_{l},g_{l}}^{\varepsilon_{l}}$ is a monomial in $F\langle X\cup X^*\rangle$, where $\varepsilon_i$ is $*$ or nothing, we define $\deg(m)=\deg(x_{i_{1},g_{1}}^{\varepsilon_{1}})\cdots
\deg(x_{i_{l},g_{l}}^{\varepsilon_{l}})$.

If we define \[(F\langle X\cup X^{*}\rangle)_{g} = \text{span}_{F}\{m=x_{i_{1},g_{1}}^{\varepsilon_{1}}\cdots x_{i_{l},g_{l}}^{\varepsilon_{l}}\,|\,
\deg(m) = g\},\]
we obtain that $F\langle X\cup X^{*}\rangle = \bigoplus_{g \in G} (F\langle X\cup X^{*}\rangle)_{g}$ is a $G$-grading on the algebra $F\langle X\cup X^{*}\rangle$, which makes it a $(G,*)$-algebra. We denote such algebra by $F\langle X\,|\, (G,*)\rangle$ and call it the free $(G,*)$-algebra. The elements of $F\langle X\,|\, (G,*)\rangle$ are called $(G,*)$-polynomials. 
 
If $A$ and $B$ are $G$-graded algebras with involution, we say that a homomorphism $\phi : A \longrightarrow B$ is a homomorphism of graded algebras with involutions, if $\phi (A_g) \subset B_{g}$, for all $g\in G$ and $\phi(x^*)=\phi(x)^*$, for all $x\in A$.

The algebra $F\langle X\,|\, (G,*)\rangle$ satisfies a universal property: for any $(G,*)$-algebra $A$ and for any map $\varphi:X\longrightarrow A$ such that for all $g\in G$, $\varphi(X_g)\subseteq A_g$, there exists a unique homomorphism of graded algebras with involution $\phi:F\langle X\,|\, (G,*)\rangle \longrightarrow A$, such that for all $x\in X$, $\phi(x)=\varphi(x)$.

Let $A$ be $(G,*)$-algebra. A polynomial $f \in F\langle X\, |\,
(G,*)\rangle$ is called a $(G,*)$-polynomial identity of $A$ if $f \in Ker(\phi)$ for any homomorphism of graded algebra with involution $\phi: F\langle X\,|\,(G,*)\rangle \rightarrow A$. Equivalently, $f$ vanishes under any admissible substitution of variables by the elements of $A$ with the condition that if $x_{k,g}$ is substituted by $a\in A_{g}$, then $x_{k,g}^*$ is substituted by $a^*$.

We observe that if $A$ is a $(G,*)$-algebra, then it is a graded
algebra, and if $f$ is a graded polynomial identity of $A$, then it is also a $(G,*)$-identity of $A$. In particular Proposition 4.1 of \cite{Diogo} also holds for $(G,*)$-algebras.

\begin{proposition}\label{elementary}
    Let $G$ be a group and let $\overline{g} = (g_{1},\ldots,g_{n}) \in
    G^{n}$ be an $n$-tuple of elements from $G$. Suppose $M_{n}(F)$
    is endowed with elementary grading induced by $\overline{g}$. The
    following assertions are equivalent
    \begin{enumerate}
        \item The neutral component of $M_{n}(F)$ coincides with the main
        diagonal.
        \item $x_{1,e}x_{2,e} - x_{2,e}x_{1,e}$ is a graded identity of
        $M_{n}(F)$.
        \item The elements of $\overline{g}$ are pairwise distinct.
    \end{enumerate}
\end{proposition}

A (two-sided) ideal $I \subset F\langle X \,|\,(G,*) \rangle$ is
called a $T_{G}^{*}$-ideal if $I$ is closed under all $(G,*)$-endomorphism of $F\langle X \,|\,(G,*)
\rangle $. We denote the set of all $(G,*)$-identities of $A$ by
$T_{G}^{*}(A)$. Let $S \subset F\langle X\,|\,(G,*)\rangle$. We
denote the intersection of all $T_{G}^{*}$-ideals containing $S$ by $\langle S \rangle_{T_{G}^{*}}$. Notice that $T_{G}^{*}(A)$ and $\langle S \rangle_{T_{G}^{*}}$ are $T_{G}^{*}$-ideals of $F\langle X\,|\,(G,*)\rangle$. We say that $S \subset F\langle X\,|\,(G,*)\rangle$ is a basis for the $(G,*)$-identities of $A$ if $T_{G}^{*}(A) = \langle S \rangle_{T_{G}^{*}}$.

\begin{proposition}\label{ids}
    Let $G$ be a group and let $M_{n}(F)$ be endowed with the elementary grading induced by
    an $n$-tuple $(g_{1},\ldots,g_{n})$ of pairwise distinct elements from $G$,
    and with the transpose involution. The following polynomials are
    $(G,*)$-identities for $M_{n}(F)$
    \begin{eqnarray}
        x_{1,e}x_{2,e} - x_{2,e}x_{1,e}\\
        x_{1,e} - x_{1,e}^{*}\\
        x_{1,g}, g\not\in Supp(M_n(F))\\
        x_{1,g}x_{2,g^{-1}}x_{3,g} - x_{3,g}x_{2,g^{-1}}x_{1,g}, g \neq e
    \end{eqnarray}
\end{proposition}

For more details about identities 1 and 4, see \cite[Lemma
4.1]{BahturinDrensky}. Identity 2 follows from Proposition
\ref{elementary}. Also, \cite[Remark 2 of Theorem 8]{Haile} shows
that identity 4 follows from identity 2.

When dealing with ordinary polynomials, it is well known that each
T-ideal is generated by its multi-homogeneous polynomials. In the
case of $(G,*)$-polynomials, we need to slight modify this concept.

\begin{definition}
    Let $f=f(x_{1,g_{1}},\ldots,x_{n,g_{n}},x_{1,g_{1}}^{*},\ldots,x_{n,g_{n}}^{*})\in F\langle X\,|\,(G,*)\rangle$. Write $f$ as
    \[f=\sum_{l=1}^{k} \lambda_{l}m_{l}\]
    where $\lambda_{l} \in F - \{0\}$ and $m_l$ are monomials in $F\langle X\,|\,(G,*)\rangle$.
    The polynomial $f$ is called \textit{strongly multi-homogeneous} if for each $t \in \{1,\ldots,m\}$,
    $\deg_{x_{t,g_{t}}} m_{i} + \deg_{x_{t,g_{t}}^{*}} m_{i} =
    \deg_{x_{t,g_{t}}} m_{j} + \deg_{x_{t,g_{t}}^{*}} m_{j}$ for all $i, j
    \in \{1,\ldots,k\}$. Here, the symbol $\deg_{x} m_{i}$
    denotes the number of times the variable $x$ appears in the monomial $m_{i}$.
\end{definition}

Following the classic Vandermonde argument, we can prove that if $I$
is a $(G,*)$-ideal, then $I$ is generated by its strongly
multi-homogeneous polynomials.

\section{The $*$-graded identities of $M_{n}(F)$}

We start this section with the following theorem of \cite{Haile}, which we aim to generalize for infinite fields and for a broader class of gradings by adding the identities $x_g=0$ for $g\not \in Supp(M_n(F))$. We observe that in \cite{Haile} the authors used the graph theory to prove this result.

\begin{theorem}[Haile-Natapov, Theorem 8, \cite{Haile}]

    Let $G = \{g_{1},\ldots,g_{n}\}$ be a group of order $n$. The ideal of $(G,*)$-identities of $M_{n}(\mathbb{C})$ endowed with the elementary grading induced by $(g_{1},\ldots,g_{n})$ and with the transpose involution is generated as a $T_{G}^{*}$-ideal by the following elements
    \begin{enumerate}
        \item $x_{i,e}x_{j,e} - x_{j,e}x_{i,e}$
        \item $x_{i,e} - x_{i,e}^{*}$
    \end{enumerate}
\end{theorem}

From now on, we consider $M_{n}(F)$ endowed with elementary grading
induced by the $n$-tuple $\overline{g} = (g_{1},\ldots,g_{n}) \in G^{n}$ of pairwise distinct elements of $G$, and we denote $G_0=Supp(M_n(F))$.

Let $g\in G$. We define
\[D(g) = \{i \in \{1,\dots, n\}\,|\, g_{i}g \in \{g_1,\dots,g_n\}\}\]
and
\[Im(g) = \{j \in \{1,\dots, n\}\,|\,g_{j}g^{-1} \in \{g_1,\dots,g_n\}\}.\]

Notice that $|D(g)| = |Im(g)|$, $D(g^{-1}) = Im(g)$ and $D(g) =
\emptyset$ if and only if $g \notin G_0$. In that case, if $i\in
D(g)$, there exists a unique $j\in \{1,\dots, n\}$ such that
$g_ig=g_j$. If we define $j=\widehat{g}(i)$, we obtain a bijective
map
\[\begin{array}{cccc}
\widehat{g}:& D(g)& \longrightarrow & Im(g) \\
& i   & \longmapsto \   & \widehat{g}(i)\\
\end{array}\]
Observe that for each $i\in \{1,\dots,n\}$, we have
$e_{i\widehat{g}(i)} \in (M_{n}(F))_{g}$ and for each $g$ in support
of $M_n(F)$, $\widehat{g^{-1}} = (\widehat{g})^{-1}$.

\begin{lemma}\label{g=h}
Let $g,h \in G$. If there exists $i \in D(g)\cap D(h)$ such that
\[\widehat{g}(i) = \widehat{h}(i),\]
then $g = h$.
\end{lemma}
\proof Let $i\in D(\hat{g})\cap D(\hat{h})$. If
$j=\hat{g}(i)=\hat{h}(i)$, then $g_ig=g_j=g_ih$. We can conclude
that $g=h$.

\endproof

Let $\Omega = \{y_{i,\widehat{g}({i})}^{k}|g \in G, i \in D(g),  k
\in \mathbb{N}\}$ be a set of commuting variables and $F[\Omega]$ be
the algebra of commuting polynomials in $\Omega$. We denote the set
of all matrices over $F[\Omega]$ by $M_{n}(\Omega)$. As in the case
of matrices over $F$, if $\overline{g}=(g_1,\dots,g_n)$ is an
$n$-tuple of elements of $G$, then $M_{n}(\Omega)$ is endowed with
an elementary $G$-grading induced by $\overline{g}$.

\begin{definition}\label{generic}
    For each $g\in G_0$ and $j\in \mathbb{N}$, the elements of $M_{n}(\Omega)$,
    \[A_{j,g} = \sum\limits_{i \in D(g)} y_{i,\widehat{g}(i)}^{j}e_{i\widehat{g}(i)}\]
    and \[A_{j,g}^{*} = \sum\limits_{i \in D(g^{-1})} y_{\widehat{g}^{-1}(i),i}^{j}e_{i\widehat{g^{-1}}(i)}\]
    are called generic $(G,*)$-matrices. The subalgebra of $M_n(\Omega)$ generated by $\{A_{j,g}, A_{j,g}^*\,|\, g\in G_0, j\in \mathbb{N}\}$ is called the algebra of generic $(G,*)$-matrices and we denote it by $Gen$.
\end{definition}

\begin{lemma}\label{lema-1}
Let $g,h\in G_0$. If $y_{i,k}^{j}\in \Omega$ is an entry of the matrices $A_{j,g}$ and $A_{j,h}$, then $g = h$.
\end{lemma}

\proof
If $y_{i,k}^{j}$ is an entry of $A_{j,g}$  and of $A_{j,g}$  then $k=\widehat{g}(i)$ and $k=\widehat{h}(i)$. Now Lemma \ref{g=h} implies that $g=h$.
\endproof

Using classical arguments, we can prove the following proposition.

\begin{proposition}
    The relatively free algebra $F\langle X\,|\,(G,*) \rangle / T_G^*(M_n(F))$ is isomorphic to
    $Gen$. Furthermore, $T_G^*(M_n(F))=T_G^*(Gen)$.
\end{proposition}

We now define the following maps, which by an abuse of notation will be also denoted by $*$ 
\[\begin{array}{l}

\begin{array}{llll}
*: &  G & \longrightarrow & G \\ & g & \longmapsto & g^*=g^{-1}\\
\end{array}
\\
\\
\begin{array}{llll} 
*: &\Omega & \longrightarrow & \Omega \\ & y_{k \widehat{g}(k)} & \longmapsto & {y_{k \widehat{g}(k)}}^*=y_{\widehat{g}^{-1}(k) k}\\
\end{array}
\end{array}\]

Given $h_1^{\varepsilon_1}, h_2^{\varepsilon_2}, \dots, h_r^{\varepsilon_r} \in G_0$, where $h_i\in G$ and $\varepsilon_i$ is $*$ or nothing, we consider the composition $\nu = \widehat{h_r^{\varepsilon_r}}\cdots\widehat{h_1^{\varepsilon_1}}$ of the corresponding functions. This may not be well defined, and we will prove in Lemma \ref{line} that in this case the monomial $x_{1,h_1}^{\varepsilon_1}\cdots x_{r,h_r}^{\varepsilon_r}$ is a graded identity for $M_n(F)$.
Otherwise, its domain $D_{\nu}=D_{\widehat{h_r^{\varepsilon_r}}\cdots\widehat{h_1^{\varepsilon_1}}}$ is the set of $i\in\{1,\dots,n\}$ for which the image $\widehat{h_r^{\varepsilon_r}}(\dots(\widehat{h_1^{\varepsilon_1}}(i))\dots)$ is well defined. 
\begin{lemma}\label{gh}
    Let $g,h \in G$, then $D(\widehat{h}\widehat{g})\subseteq D(\widehat{gh})$. Moreover, if $i\in D(\widehat{h}\widehat{g})$, then $\widehat{h}\widehat{g}(i)=\widehat{gh}(i)$.
\end{lemma}

\proof If $D(\widehat{h}\widehat{g})=\emptyset$, the result is
obvious. Suppose $D(\widehat{h}\widehat{g})\neq \emptyset$. If $i\in
D(\widehat{h}\widehat{g})$, let $k=\widehat{g}(i)$ and
$j=\widehat{h}(k)$. Then, $g_{k}=g_ig$ and $g_j=g_kh$, and we obtain
$g_j=g_i(gh)$, that is, $\widehat{gh}(i)=j$.
\endproof

\begin{lemma}\label{line}
    Let $h_1^{\varepsilon_1},h_2^{\varepsilon_2},\dots, h_r^{\varepsilon_r}\in G_0$.
    If $D_{\widehat{h_r^{\varepsilon_r}}\cdots\widehat{h_1^{\varepsilon_1}}}=\emptyset$ then
\begin{center}
    $A_{i_1,h_1}^{\varepsilon_1}A_{i_2,h_2}^{\varepsilon_2}\cdots A_{i_r,h_r}^{\varepsilon_r}=0$.
\end{center}
    Moreover, if the set $D_{\widehat{h_r^{\varepsilon_r}}\cdots\widehat{h_1^{\varepsilon_1}}}$ is nonempty then the $i$-th line
    of the matrix  $A_{i_1,h_1}^{\varepsilon_1}A_{i_2,h_2}^{\varepsilon_2}\cdots A_{i_r,h_r}^{\varepsilon_r}$ is nonzero if and
    only if $i\in D_{\widehat{h_r^{\varepsilon_r}}\cdots\widehat{h_1^{\varepsilon_1}}}$.
    In this case, if $j=\widehat{h_r^{\varepsilon_r}}\cdots \widehat{h_1^{\varepsilon_1}}(i)$,
    the only nonzero entry in the $i$-th line is a monomial of $\Omega$ in the $j$-th column.
\end{lemma}
\begin{proof}
The proof is by induction on the length $r$ of the product. The
result for $r=1$ follows directly from Definition \ref{generic}.
Hence, we consider $r>1$ and assume the result for products of
length $r-1$. Let us consider the first case
$D_{\widehat{h_{r}^{\varepsilon_r}}\cdots\widehat{h_1^{\varepsilon_1}}}\neq\emptyset$.
In this case
$D_{\widehat{h_{r-1}^{\varepsilon_{r-1}}}\cdots\widehat{h_1^{\varepsilon_1}}}\neq\emptyset$
and we denote
$\nu=\widehat{h_{r-1}^{\varepsilon_{r-1}}}\cdots\widehat{h_1^{\varepsilon_1}}$.
The induction hypothesis implies that there exists monomials $m_i$,
where $i \in
D_{\widehat{h_{r-1}^{\varepsilon_{r-1}}}\cdots\widehat{h_1^{\varepsilon_1}}}$,
such that

\begin{equation}\label{a}
A_{i_1,h_1}^{\varepsilon_1}A_{i_2,h_2}^{\varepsilon_2}\cdots
A_{i_r,h_r}^{\varepsilon_r} = \left(\sum_{i \in
D_{\widehat{h_{r-1}^{\varepsilon_{r-1}}}\cdots\widehat{h_1^{\varepsilon_1}}}}m_i
e_{i \nu(i)}\right) \left(\sum_{j \in
D_{\widehat{h_r^{\varepsilon_r}}}}(y_{j\widehat{h_r}(j)})^{\varepsilon_r}
e_{j\widehat{h_r^{\varepsilon_r}}(j)}\right)
\end{equation}

Note that $e_{i \nu(i)} e_{j\widehat{h_r^{\varepsilon_r}}(j)}\neq 0$
for some $j$ if and only if $i \in
D_{\widehat{h_{r}^{\varepsilon_r}}\cdots\widehat{h_1^{\varepsilon_1}}}$.
In this case, the product equals $e_{i
\widehat{h_r^{\varepsilon_r}}(j)}$. Hence, we obtain
\[A_{i_1,h_1}^{\varepsilon_1}A_{i_2,h_2}^{\varepsilon_2}\cdots
A_{i_r,h_r}^{\varepsilon_r}=\sum_{i\in
D_{\widehat{h_{r}^{\varepsilon_r}}\cdots\widehat{h_1^{\varepsilon_1}}}}
(m_i(y_{j\widehat{h_r}(j)})^{\varepsilon_r})  e_{i
\widehat{h_r^{\varepsilon_r}}(\nu(i))},\] and the result follows.
Now, assume that
$D_{\widehat{h_{r}^{\varepsilon_r}}\cdots\widehat{h_1^{\varepsilon_r}}}=\emptyset$.
If
$D_{\widehat{h_{r-1}^{\varepsilon_{r-1}}}\cdots\widehat{h_1^{\varepsilon_1}}}=\emptyset$
then by the induction hypothesis
$A_{i_1,h_1}^{\varepsilon_1}A_{i_2,h_2}^{\varepsilon_2}\cdots
A_{i_{r-1},h_{r-1}}^{\varepsilon_{r-1}}=0$ and the result holds.
Moreover, if
$D_{\widehat{h_{r-1}^{\varepsilon_{r-1}}}\cdots\widehat{h_1^{\varepsilon_1}}}\neq\emptyset$
then we may write the product
$A_{i_1,h_1}^{\varepsilon_1}A_{i_2,h_2}^{\varepsilon_2}\cdots
A_{i_{r},h_{r}}^{\varepsilon_{r}}$ as in (\ref{a}). Since
$D_{\widehat{h_{r}^{\varepsilon_r}}\cdots\widehat{h_1^{\varepsilon_r}}}=\emptyset$,
every product $e_{i \nu(i)} e_{j\widehat{h_r^{\varepsilon_r}}(j)}$
equals zero and therefore
$A_{i_1,h_1}^{\varepsilon_1}A_{i_2,h_2}^{\varepsilon_2}\cdots
A_{i_r,h_r}^{\varepsilon_r}=0$.
\end{proof}

\begin{definition}\label{seq}
Suppose ${\bf h}=(h_{{1}}^{\epsilon_{1}},\ldots,h_{{m}}^{\epsilon_{m}})\in G^m$ such that $D_{\widehat{h_{m}^{\varepsilon_m}}\cdots\widehat{h_1^{\varepsilon_r}}}\neq \emptyset$.

For each $k \in D_{\widehat{h_{m}^{\varepsilon_m}}\cdots\widehat{h_1^{\varepsilon_r}}}$, we denote by
$s_{k}({\bf h}) =
(s_{1}^{k},\ldots,s_{m}^{k},s_{m+1}^{k})$ the following sequence,
inductively by setting:
\begin{description}
    \item (1) $s_{1}^{k} = k$
    \item (2) $s_{r}^{k} = \widehat{h_{r-1}^{\varepsilon_{r -
                1}}}(s_{r-1}^k)$, for $r \in \{2,\ldots,m+1\}$.
\end{description}
We denote by
$t_{k}({\bf h}) =(t_{1}^{k},\ldots,t_{m}^{k})$ the sequence defined by \[t_{r}^{k} = \widehat{h_{r}}(s_r^k),\quad r \in \{1,\ldots,m\}\]
\end{definition}

\begin{lemma}\label{lema0}
Let $h_{{1}}^{\varepsilon_{1}},\ldots,h_{{m}}^{\varepsilon_{m}}\in G
$ such that $D_{\widehat{h_{m}^{\varepsilon_m}}\cdots\widehat{h_1^{\varepsilon_1}}}\neq \emptyset$. Then
\begin{center}
$A_{1,h_{1}}^{\varepsilon_{1}}\cdots A_{m,h_{m}}^{\varepsilon_{m}}
= \sum\limits_{k \in D_{\widehat{h_{m}^{\varepsilon_m}}\cdots\widehat{h_1^{\varepsilon_1}}}}
\omega_{k}^me_{k,s_{m+1}^{k}}$,
\end{center}
where $\omega_{k}^m =
(y_{s_{1}^{k},t_{1}^{k}}^{1})^{\varepsilon_{1}}\cdots
(y_{s_{m}^{k},t_{m}^{k}}^{m})^{\varepsilon_{m}}$. Furthermore, each matrix in the product $(A_{1,h_{1}}^{\varepsilon_{1}},\ldots,A_{m,h_{m}}^{\varepsilon_{m}})$ contributes with exactly one factor of it in the product
$(y_{s_{1}^{k},t_{1}^{k}}^{1})^{\varepsilon_{1}}\cdots (y_{s_{m}^{k},t_{m}^{k}}^{m})^{\varepsilon_{m}}$. For each $p \in \{1,\ldots,m\}$, $A_{p,h_{p}}^{\varepsilon_{p}}$ contributes with $(y_{s_{p}^{k},t_{p}^{k}}^{p})^{\varepsilon_{p}}$.
\end{lemma}

\begin{proof}

The proof follows by induction on $m$. If $m=1$, the result is
obvious.

Suppose $m>1$. By the induction hypothesis, we obtain

\begin{align*}
A_{1,h_{1}}^{\varepsilon_{1}}\cdots A_{m,h_{m}}^{\varepsilon_{m}}
 & = (A_{1,h_{1}}^{\varepsilon_{1}}\cdots
A_{m-1,h_{m-1}}^{\varepsilon_{m-1}})A_{m,h_{m}}^{\varepsilon_{m}}\\ & = \sum\limits_{k \in
D_{\widehat{h_{m-1}^{\varepsilon_{m-1}}}\cdots\widehat{h_1^{\varepsilon_1}}}}
\omega_{k}^{m-1}e_{k,s_{m}^{k}}\sum\limits_{i \in
D(\widehat{h_m^{\varepsilon_m}})}
(y_{i,\widehat{h_m}(i)}^{m})^{\varepsilon_m}e_{i\widehat{h_m}(i)} \\ & = \sum\limits_{k \in
D_{\widehat{h_{m-1}^{\varepsilon_{m-1}}}\cdots\widehat{h_1^{\varepsilon_1}}}}
\omega_{k}^{m-1}(y^m_{s_m^k,\widehat{h_m}(s_m^k)})^{\varepsilon_m}e_{k,s_{m+1}^{k}}
\end{align*}

Now, the proof follows once one observes that $t_{m}^{k} =
\widehat{h_{m}}(s_m^k)$ and
$\omega_{k}^{m}=\omega_{k}^{m-1}(y^m_{s_m^k,t_m^k})^{\varepsilon_m}$.
\end{proof}

\begin{definition}
    Let $\sigma \in S_{m}$. For $m = x_{i_{\sigma(1)},h_{\sigma(1)}}^{\varepsilon_{\sigma(1)}}\ldots x_{i_{\sigma(n),h_{\sigma(n)}}}^{\varepsilon_{\sigma(n)}}$, and any two integers $1 \leq k \leq l \leq n$, we denote $m^{[k,l]}$ the subword obtained from $m$ by deleting  the first $k - 1$ and the last $m - l$ variables.
    \begin{center}
        $m^{[k,l]} = x_{i_{\sigma(k),h_{\sigma(k)}}}^{\varepsilon_{\sigma(k)}}\ldots
        x_{i_{\sigma(l),h_{\sigma(l)}}}^{\varepsilon_{\sigma(l)}}$.
    \end{center}
\end{definition}

\begin{lemma}\label{tecnico}
    Let $x_{i_1,h_1}^{\varepsilon_1}\cdots x_{i_r,h_r}^{\varepsilon_r}$ and $x_{j_1,h_1'}^{\eta_1}\cdots x_{j_l,h_l'}^{\eta_l}$ be two monomials, with $\varepsilon_k$ and $\eta_k$ being $*$ or nothing, such that the matrices $A_{i_1,h_1}^{\varepsilon_1}\cdots A_{i_r,h_r}^{\varepsilon_r}$ and $A_{j_1,h_1'}^{\eta_1}\cdots A_{j_s,h_s'}^{\eta_s}$ have in the same position, the same nonzero entry.
    Then, $r=l$ and there exists a permutation $\sigma \in S_{r}$ such that $j_{q} = i_{\sigma(q)}$ and $h_{q}' = h_{\sigma(q)}$ for all $q \in \{1,\dots, r\}$. In particular, $x_{i_1,h_1}^{\varepsilon_1}\cdots x_{i_r,h_r}^{\varepsilon_r} - x_{j_1,h_1'}^{\eta_1}\cdots x_{j_l,h_l'}^{\eta_l}$ is a strongly multi-homogeneous polynomial.
    If this entry is $(y_{{s}_{1}^{k},{t}_{1}^{k}}^{i_1})^{\varepsilon_{1}}\cdots (y_{{s}_{r}^{k},{t}_{r}^{k}}^{i_r})
    ^{\varepsilon_{r}}$, then $(y_{{s}_{\sigma(q)}^{k},{t}_{\sigma(q)}^{k}}^{i_\sigma(q)})^{\varepsilon_{\sigma(q)}} =  (y_{{s'}_{q}^{k},{t'}_{q}^{k}}^{j_q})^{\eta_{q}}$
    for $q = 1, \ldots, r$.
\end{lemma}

\begin{proof}

Let $A_{i_1,h_{1}}^{\varepsilon_{1}}\cdots
A_{i_r,h_{r}}^{\varepsilon_{r}} = \sum\limits_{k \in
D_{\widehat{h_{m}^{\varepsilon_r}}\cdots\widehat{h_1^{\varepsilon_1}}}}
\omega_{k}^re_{k,s_{r+1}^{k}}$, and $A_{j_1,h_{1}'}^{\eta_{1}}\cdots
A_{j_l,h_{l}'}^{\eta_{l}} = \sum\limits_{k \in
D_{\widehat{h_{l}'^{\varepsilon_l}}\cdots\widehat{h_1'^{\varepsilon_1}}}}
\tilde{\omega}_{k}^le_{k,s_{m+1}^{k}}$ as in Lemma \ref{lema0}. Let
$k$ be the row in which these two matrices have the same nonzero
entry. Let ${\bf h}=(h_1,\dots,h_r)$ and ${\bf
h'}=(h_1',\dots,h_l')$ and consider the sequences $s_k({\bf
h})=(s_1^k,\dots, s_{r+1}^k)$, $s_k({\bf h'})=({s'}_1^{k},\dots,
{s'}_{r+1}^{k})$, $t_k({\bf h})=(t_1^k,\dots, t_{r+1}^k)$, $t_k({\bf
h'})=({t'}_1^{k},\dots, {t'}_{r+1}^{k})$ as in Definition \ref{seq}.
Then $\omega_{k}^r=\tilde{\omega}_{k}^l$, where
 $\omega_{k}^r = (y_{s_{1}^{k},t_{1}^{k}}^{i_1})^{\varepsilon_{1}}\cdots
 (y_{s_{r}^{k},t_{r}^{k}}^{i_r})^{\varepsilon_{r}}$ and  $\tilde\omega_{k}^l =(y_{{s'}_{1}^{k},{t'}_{1}^{k}}^{j_1})^{\eta_{1}}\cdots
  (y_{{s'}_{l}^{k},{t'}_{l}^{k}}^{j_l})^{\eta_{l}}$.
  Of course, we have $r=l$, and for each $q\in \{1\dots,r\}$, there exists $p\in \{1\dots,r\}$ such
  that $(y_{{s'}_{q}^{k},{t'}_{q}^{k}}^{j_q})^{\eta_{q}} = (y_{{s}_{p}^{k},{t}_{p}^{k}}^{i_p})^{\varepsilon_{p}}$.

Let us now consider two cases:

Case 1: $\varepsilon_p=\eta_q$. Then $i_p=j_q$, $s_p^k={s'}_q^k$ and
$t_p^k={t'}_q^k$. Since ${t'}_q^k=\widehat{h_q'}({s'}_q^k)$ and
${t}_p^k=\widehat{h_p}({s}_p^k)$, we obtain
$\widehat{h_p}(s_p^k)=\widehat{h'_q}(s_p^k)$ and Lemma \ref{lema-1}
implies that $h_p=h_q'$.

Case 2: $\varepsilon_p\neq\eta_q$. We suppose $\varepsilon_p=*$.
Then, we have
\[(y_{{s'}_{q}^{k},\widehat{h'_q}({s'}_{q}^{k})}^{j_q})=(y_{{s}_{p}^{k},\widehat{h_p}({s}_{p}^{k})}^{i_p})^*=y_{\widehat{h_p^{-1}}({s}_{p}^{k}),{s}_{p}^{k}}^{i_p}.\]
By comparing the indexes, we obtain $i_p=j_q$, $s_p^k=\widehat{h'_q}({s'}_q^k)$ and ${s'}_q^k=\widehat{h_p^{-1}}(s_p^k)$.

Hence ${s'}_q^k=\widehat{h_p^{-1}}(\widehat{h'_q}({s'}_q^k))$ and by Lemmas \ref{g=h} and \ref{gh}, we obtain $h'_qh_p^{-1}=e$ and $h'_q=h_p$.

From the above, we conclude that $x_{i_1,h_1}^{\varepsilon_1}\cdots x_{i_r,h_r}^{\varepsilon_r} - x_{j_1,h_1'}^{\eta_1}\cdots x_{j_l,h_l'}^{\eta_l}$
is a strongly multi-homogeneous polynomial.

\end{proof}

\begin{remark}
Suppose that the same entry is $(k,l)$. Notice that there exist
matrix units $e_{a_{1}b_{1}} \in M_{n}(F)_{\alpha(x_{j_{1},h_{1}'})}
,\ldots, e_{a_{r}b_{r}} \in M_{n}(F)_{\alpha(x_{j_{r},h_{r}'})}$
such that
$(e_{a_{1}b_{1}})^{\eta_{1}}\ldots(e_{a_{r}b_{r}})^{\eta_{r}} =
e_{kl}$ and
$(e_{a_{\sigma(p)}b_{\sigma(p)}})^{\varepsilon_{\sigma(p)}} =
(e_{a_{p}b_{p}})^{\eta_{p}}$ for $p = 1,\ldots,r$.
\end{remark}

\section{The main theorem}

We denote by $J$ the $T_{G}^{*}$-ideal generated by the polynomials
\begin{center}
$x_{1,e}x_{2,e} - x_{2,e}x_{1,e}$ and  $x_{1,e} - x_{1,e}^{*}$.
\end{center}

In next two lemmas, we follow the ideas of \cite{Azevedo},
\cite{Diogo}, and \cite{Vasilovsky2}.

\begin{lemma}\label{transposicao1.5}
Let $m_{1}(x_{i_{1},h_1}^{\varepsilon_{1}},\ldots,x_{i_{r},h_r}^{\varepsilon_{r}}),m_{2}(x_{i_{1},h_1}^{\eta_{1}},\ldots,x_{i_{l},h_l}^{\eta_{l}})$ be two monomials that start with the same variable and let $\overline{m_{1}}$ and $\overline{m_{2}}$ be the monomials obtained from $m_{1}$ and $m_{2}$ by deleting the first variable.

If $m_{1}(A_{i_{1},h_1}^{\varepsilon_{1}},\ldots,A_{i_{r},h_r}^{\varepsilon_{r}})$ and $m_{2}(A_{i_{1},h_1}^{\eta_{1}},\ldots,A_{i_{l},h_l}^{\eta_{l}})$ have in the same position the same non-zero entry, then
$\overline{m_{1}}(A_{i_{1},h_1}^{\varepsilon_{1}},\ldots,A_{i_{r},h_r}^{\varepsilon_{r}})$ and $\newline \overline{m_{2}}(A_{i_{1},h_1}^{\eta_{1}},\ldots,A_{i_{l},h_l}^{\eta_{l}})$ have in the same position the same non-zero entry.
\end{lemma}

\begin{proof}
It follows from Lemma \ref{lema0}.
\end{proof}

\begin{lemma}\label{transposicao2}
Let $m_{1}=x_{i_1,h_1}^{\varepsilon_{1}}\cdots
x_{i_r,h_r}^{\varepsilon_{r}}$ and $m_{2}=x_{j_1,h'_1}^{\eta_{1}}
\cdots x_{j_l,h'_l}^{\eta_{l}}$ be two monomials such that
\begin{center}
$A_{i_{1},h_1}^{\varepsilon_{1}} \cdots
A_{i_{r},h_r}^{\varepsilon_{r}}$ and
$A_{j_{1},h'_1}^{\eta_{1}}\cdots A_{j_{l},h'_l}^{\eta_{l}}$
\end{center}
have in the same position, the same non-zero entry. Then $m_{1}
\equiv m_{2} \mod  J$.
\end{lemma}
\begin{proof}

We prove this lemma by induction on $n$.

Suppose $n=1$. If $A_{i_1,h_1}^{\varepsilon_{1}}$ and
$A_{j_1,h'_1}^{\eta_{1}}$ have in the position $(p,q)$ the same nonzero entry, then by Lemma \ref{tecnico}, $i_1=j_1$, and
$h'_1=h_1$. If $\varepsilon_{1}=\eta_{1}$, then $m_1=m_2$ and they are equivalent modulo $J$. If $\varepsilon_{1}\neq \eta_{1}$, by comparing the $(p,q)$ entries, we have $y_{p,\widehat{h_1}(p)}^{i_1} = y_{\widehat{h_{1}^{-1}}(p),p}^{i_1}$. Then, $\widehat{h_1}(p)= p$ and by Lemma \ref{g=h}, we obtain $h = e$, the neutral element of $G$. Hence, they are equivalent modulo $J$.

In proving the inductive step, we will show that $m_{2}$ is
congruent, modulo $J$ to a monomial $m_{3}$ that starts with the
same variable of $m_{1}$. Therefore, $m_{1}$ and $m_{3}$ will have in the same position, the same non-zero entry. According to Lemma \ref{transposicao1.5}, $\overline{m_{1}}$ and $\overline{m_{3}}$ have in the same position the same non-zero entry. Thus, by induction,
$m_{1} \equiv m_{3} \ \ mod \ \ J$, and consequently, $m_{1} \equiv
m_{3} \equiv m_{2} \ \ mod \ \ J$.

According to Lemma \ref{tecnico}, $m_{1} - m_{2}$ is a strongly
multi-homogeneous polynomial. Furthermore, there exists a
permutation $\sigma \in S_{l}$ such that $h_{i_{\sigma(s)}} =
h_{i_{s}}', i_{\sigma(s)} = j_{s}$, for all $s \in \{1,\ldots,l\}$.

From Remark of Lemma \ref{tecnico}, there exist matrix units
\begin{center}
    $e_{a_{1}b_{1}} \in M_{n}(F)_{\alpha(x_{j_{1},h_{1}'})},
    \ldots, e_{a_{l}b_{l}} \in M_{n}(F)_{\alpha(x_{j_{l},h_{l}'})}$
\end{center}
such that  $(e_{a_{1}b_{1}})^{\eta_{1}}\cdots
(e_{a_{l}b_{l}})^{\eta_{l}} = e_{pq}$. Furthermore,
\begin{center}
    $(e_{a_{\sigma(u)}b_{\sigma(u)}})^{\varepsilon_{\sigma(u)}}=
    (e_{a_{u}b_{u}})^{\eta_{u}}$ for $u = 1,\ldots,l$.
\end{center}

Suppose that the same entry of assumption is
$(y_{{s}_{1}^{k},{t}_{1}^{k}}^{i_1})^{\varepsilon_{1}}\cdots
(y_{{s}_{l}^{k},{t}_{l}^{k}}^{i_l})^{\varepsilon_{l}}$ and the same
position is $(p,q)$. Assume $\sigma^{-1}(1) = 1$. Note that
$(y_{{s}_{1}^{k},{t}_{1}^{k}}^{i_1})^{\varepsilon_{1}} =
(y_{{s'}_{1}^{k},{t'}_{1}^{k}}^{j_1})^{\eta_{1}}, j_{1} = i_{1},$
and $h_{1} = h_{1}'$. The letters
$(y_{{s}_{1}^{k},{t}_{1}^{k}}^{i_1})^{\varepsilon_{1}} =
(y_{{s'}_{1}^{k},{t'}_{1}^{k}}^{j_1})^{\eta_{1}}$ will appear in the
$p$-th row of $A_{i_{1},h_1}^{\varepsilon_{1}}$ and
$A_{j_{1},h'_1}^{\eta_{1}}$. Therefore, $\varepsilon_1=\eta_1$ or
$\varepsilon_{1} \neq \eta_{1}$ and $h_{1} = h_{1}' = e$. The
analysis of first situation is immediate. In the second, we have
$m_{2} \equiv (x_{j_{1},h_{1}'}^{\eta_{1}})^{*}m_{2}^{[2,l]}$ by
identity $x_{1,e} - x_{1,e}^{*}$.

Now, suppose that $\sigma^{-1}(1) > 1$. To analyze the monomial
$m_{2}$, we denote the number $\sigma^{-1}(1)$ by $k_{2}$. Let $t$
be the least positive integer such that $\sigma^{-1}(t+1) <
\sigma^{-1}(1) \leq \sigma^{-1}(t)$. We denote the number
$\sigma^{-1}(t+1)$ by $k_{1}$ and the number $\sigma^{-1}(t)$ by
$k_{3}$. 
We divide the rest of proof into $4$ cases.

\vspace{0,5cm}

Case 1. $\eta_{1} = \mbox{nothing}$ and $\varepsilon_{1} = *$. In
this situation, $b_{1} = p$. If $\varepsilon_{\sigma(1)} = *$, then
$b_{\sigma(1)} = p$. If $\varepsilon_{\sigma(1)} = \mbox{nothing}$,
then $a_{\sigma(1)} = p$. Note that $\alpha(m_{2}^{[1,k_{2}]}) =
\alpha(e_{a_{\sigma(1)}b_{\sigma(1)}}^{\varepsilon_{\sigma(1)}}\ldots
e_{a_{\sigma(k_{2})}b_{\sigma(k_{2})}}^{\varepsilon_{\sigma(k_{2})}})
= \alpha(e_{pp}) = e$. Therefore,
\begin{center}
$m_{2} \equiv m_{4} =
x_{j_{k_{2}},h_{k_{2}}'}(m_{2}^{[1,k_{2}-1]})^{*}(m_{2}^{[k_{2}+1,l]})
\ \ mod \ \ J$.
\end{center}
The last equivalence follows from identity $x_{1,e} - x_{1,e}^{*}$.
If $h_{k_{2}}' = h_{1} = e$, the result follows from identity
$x_{1,e} - x_{1,e}^{*}$. Observe that the variables
$x_{i_{1},h_{1}}^{\varepsilon_{1}}$ and $x_{i_{1},h_{1}} =
x_{j_{k_{2}}h_{k_{2}}'}$ could contribute with the same letter of
$F[\Omega]$ in entry $(p,q)$. If this occurs, then $h_{k_{2}}' =
h_{1} = e$. From $x_{1,e} - x_{1,e}*$, we obtain the desired result.
Otherwise, suppose that $h_{1} \neq e$,
$x_{i_{1}h_{1}}^{\varepsilon_{1}}$ and $x_{j_{k_{2}},h_{k_{2}}'}$ do
not contribute with the same letter of $F[\Omega]$. Thus, there
exists an integer $w, w \neq k_{2}, 1 < w \leq l,$ such that
$x_{j_{w},h_{w}'}^{\eta_{w}}$ and
$x_{i_{1},h_{1}}^{\varepsilon_{1}}$ contribute with the same letter of $F[\Omega]$ in entry $(p,q)$. Without loss of generality, suppose that $w < k_{2}$. Notice that $\varepsilon_{1} \neq \eta_{w}$ and $\alpha(m_{4}^{[1,w+1]}) = e$. Hence
\begin{center}
$m_{2} \equiv m_{4} \equiv
x_{i_{1},h_{1}}^{\varepsilon_{1}}(m_{4}^{[1,w]})^{*}m_{4}^{[w+2,l]}
\ \ mod \ \ J$.
\end{center}
The Case 1 is verified.

\vspace{0,5cm}

Case 2. $\varepsilon_{1} = \mbox{nothing}$ and $\eta_{1} = *$. It is
analogous to case 2.

\vspace{0,5cm}

Case 3. $\eta_{1} = \varepsilon_{1} = *$. Now,
$\alpha(m_{2}^{[1,k_{2}-1]}) = g_{b_{1}}g_{b_{1}}^{-1} = e$. Suppose
that $\eta_{t} = \eta_{t+1} = \mbox{nothing}$ (the other three cases
are analogous). We analyze eight subcases.

\begin{description}
\item Case 3.1: $k_{1} = 1$.
\subitem Subcase 3.1.1: $\varepsilon_{t} \neq \eta_{t}$ and
$\varepsilon_{t + 1} = \eta_{t+1}$. Here,
$\alpha(m_{2}^{[1,k_{2}-1]})
\newline = \alpha(m_{2}^{[k_{2},k_{3}-1]}) = e$. Consequently, by
identity $x_{1,e}x_{2,e} - x_{2,e}x_{1,e}$, we have $m_{2} \equiv
m_{2}^{[k_{2},k_{3}-1]}m_{2}^{[1,k_{2}-1]}m_{2}^{[k_{3},l]} \ \ mod
\ \ J$. \subitem Subcase 3.1.2: $\varepsilon_{t} = \eta_{t}$ and
$\varepsilon_{t + 1} = \eta_{t+1}$. This subcase is similar to
Subcase 3.1.1. Here $\alpha(m_{2}^{[1,k_{2}-1]}) =
\alpha(m_{2}^{[k_{2},k_{3}]}) = e$.

\subitem Subcase 3.1.3: $\varepsilon_{t} \neq \eta_{t}$ and
$\varepsilon_{t + 1} \neq \eta_{t+1}$. Here, $\alpha(m_{2}^{[1,1]})
= \alpha(m_{2}^{[2,k_{2}-1]})^{-1} =
\alpha(m_{2}^{[k_{2},k_{3}-1]})$. Hence, by identity $x_{1,e} -
x_{1,e}^{*}$, we have \newline $m_{2} \equiv
m_{2}^{[k_{2},k_{3}-1]}(m_{2}^{[1,1]})^{*}(m_{2}^{[2,k_{2}-1]})^{*}m_{2}^{[k_{3},l]}
\ \ mod \ \ J$.

\subitem Subcase 3.1.4: $\varepsilon_{t} = \eta_{t}$ and
$\varepsilon_{t + 1} \neq \eta_{t+1}$ . This subcase is analogous to
Subcase 3.1.3. Now, $\alpha(m_{2}^{[1,1]}) =
\alpha(m_{2}^{[2,k_{2}-1]})^{-1} =
\newline \alpha(m_{2}^{[k_{2},k_{3}]})$.

\item Case 3.2: $k_{1} > 1$.

\subitem Subcase 3.2.1: $\varepsilon_{t} = \eta_{t}$ and
$\varepsilon_{t + 1} = \eta_{t+1}$. Here,
$\alpha(m_{2}^{[1,k_{1}-1]}) \newline =
\alpha(m_{2}^{[k_{1},k_{2}-1]})^{-1} =
\alpha(m_{2}^{[k_{2},k_{3}]})$. Thus, by identity $x_{1,e} -
x_{1,e}^{*}$, we have $m_{2} \equiv
m_{2}^{[k_{2},k_{3}]}(m_{2}^{[1,k_{1}-1]})^{*}(m_{2}^{[k_{1},k_{2}-1]})^{*}m_{2}^{[k_{3}+1,l]}
\ \ mod \ \ J$.

\subitem Subcase 3.2.2: $\varepsilon_{t} \neq \eta_{t}$ and
$\varepsilon_{t + 1} = \eta_{t+1}$. Here,
$\alpha(m_{2}^{[1,k_{1}-1]})
\newline = \alpha(m_{2}^{[k_{1},k_{2}-1]})^{-1} =
\alpha(m_{2}^{[k_{2},k_{3}-1]})$. Thus, by identity $x_{1,e} -
x_{1,e}^{*}$, we have $m_{2} \equiv
m_{2}^{[k_{2},k_{3}-1]}(m_{2}^{[1,k_{1}-1]})^{*}(m_{2}^{[k_{1},k_{2}-1]})^{*}m_{2}^{[k_{3},l]}
\ \ mod \ \ J$.

\subitem Subcase 3.2.3: $\varepsilon_{t} = \eta_{t}$ and
$\varepsilon_{t + 1} \neq \eta_{t+1}$. In this subcase,
$\alpha(m_{2}^{[1,k_{1}]}) \newline =
\alpha(m_{2}^{[k_{1}+1,k_{2}-1]})^{-1} =
\alpha(m_{2}^{[k_{2},k_{3}]})$. Thus, by identity $x_{1,e} -
x_{1,e}^{*}$, we have $m_{2} \equiv
m_{2}^{[k_{2},k_{3}]}(m_{2}^{[1,k_{1}]})^{*}(m_{2}^{[k_{1}+1,k_{2}-1]})^{*}m_{2}^{[k_{3},l]}
\ \ mod \ \ J$.

\subitem Subcase 3.2.4: $\eta_{t} \neq \varepsilon_{t}$ and
$\eta_{t+1} \neq \varepsilon_{t + 1}$. Finally,
\newline $\alpha(m_{2}^{[1,k_{1}]}) =
\alpha(m_{2}^{[k_{1}+1,k_{2}-1]})^{-1} =
\alpha(m_{2}^{[k_{2},k_{3}-1]})$. In this way, by identity $x_{1,e}
- x_{1,e}^{*}$, we have $n \equiv
m_{2}^{[k_{2},k_{3}-1]}(m_{2}^{[1,k_{1}]})^{*}(m_{2}^{[k_{1}+1,k_{2}-1]})^{*}m_{2}^{[k_{3},l]}
\newline mod \ \ J$.
\end{description}

\vspace{0,5cm}

Case 4. $\eta_{1} = \varepsilon_{1} = \mbox{nothing}$. It is similar
to case 3.

\end{proof}

We now recall the result \cite[Corollary 3.2]{DiogoThiago}, which is
based on an idea of \cite[Corollary 11]{CentroneMello} about graded
monomial identities.

\begin{lemma}\label{gradedmonomial}
    If a monomial $x_{i_1,h_1}\dots x_{i_p,h_p}$ in $F\langle X \rangle$ is a graded identity for $M_n(F)$, then it is a consequence of a monomial in $T_G(M_n(F))$ of length at most $2n-1$.
\end{lemma}

By Lemma 3.7, a monomial $x_{i_1,h_1}^{\varepsilon_1}\cdots
x_{i_r,h_r}^{\varepsilon_r}$ is a $(G,*)$-identity for $M_n(F)$ if
and only if $D_{\widehat{h_r^{\varepsilon_r}}\cdots
\widehat{h_1^{\varepsilon_1}}}=\emptyset$. In particular, we obtain
the following lemma.

\begin{lemma}
    A monomial $x_{i_1,h_1}^{\varepsilon_1}\cdots x_{i_r,h_r}^{\varepsilon_r}$ is a $(G,*)$-identity for $M_n(F)$ if and only if $x_{i_1,h_1^{\varepsilon_1}}\cdots x_{i_r,h_r^{\varepsilon_r}}$ is a $G$-graded identity for $M_n(F)$.
\end{lemma}

The following proposition is a straightforward consequence of the
above lemmas.

\begin{proposition}\label{monomial5}
Let $m$ be a monomial identity of $M_{n}(F)$. Then, $m$ is a
consequence of monomial identities of degree up to $2n-1$.
\end{proposition}

\begin{remark}
We recall that in \cite{CentroneMello} and \cite{DiogoThiago}, the
authors conjecture that the graded monomial identities of $M_n(F)$
follow from the graded identities of degree up to $n$. We observe
that once this conjecture is true, the same holds for the
$(G,*)$-identities of $M_n(F)$.
\end{remark}

We now state the main theorem of this paper.

\begin{theorem}
    Let $U$ be the $T_{G}^{*}$-ideal generated by identities
    $(1),(2),(3),$ and by the $(G,*)$-monomial identities of degree up to $2n-1$ of $M_n(F)$.
    Then,
    \[T_G^*(M_n(F))=U.\]
\end{theorem}

\begin{proof}
From Proposition \ref{ids}, $U \subseteq T_{G}^{*}(M_{n}(F))$. Let
us suppose \newline $U \subsetneqq T_{G}^{*}(M_{n}(F))$ Then, there
exists a strongly multi-homogeneous polynomial $\newline f\in
T_G^*(M_n(F))-U$. By writing $f=\sum_{i=1}^l\lambda_im_i$, we may suppose that all $\lambda_{i} \in F - \{0\}$, $m_{i} \in
F\langle X |(G,*)\rangle$ are monomials, which are not $(G,*)$-identities for $M_n(F)$ and that the number $l$ of nonzero summands of $f$ is minimal among the strongly multi-homogeneous polynomials $f\in T_G^*(M_n(F))-U$.

Since $f\in T_{G}^{*}(M_n(F))$, we have
\[\lambda_1 m_1\equiv \sum_{i=2}^l-\lambda_i m_i \mod U.\]
By substituting the variables by generic matrices, we obtain that $m_1$ and some $m_j$, $j>1$, have the same nonzero entry in the same position. Hence, by Lemma \ref{transposicao2}, we conclude that $m_{1} \equiv
m_{j}  \mod U$.

Now let \[h = (\lambda_{j}+\lambda_1)m_1 + \sum_{i\not\in \{1,j\}} \lambda_im_i\] Then $h\equiv f \mod U$ and the number of non-zero summands of $h$ is $l-1<l$. This is a contradiction.
\end{proof}

\begin{flushleft}
    \textbf{Acknowledgements}
\end{flushleft}
This work was started when the authors were visiting IMPA in (brazilian) summer 2016. The authors would like to thank IMPA for the hospitality and for the financial support.

\end{document}